\documentclass[reqno]{amsart}

\usepackage{amsfonts,amsthm,amsmath,amssymb,amscd,epsf,latexsym,mathrsfs}
\usepackage{graphicx}

\newtheorem{theorem}{Theorem}[section]
\newtheorem{proposition}[theorem]{Proposition}
\newtheorem{lemma}[theorem]{Lemma}
\newtheorem{corollary}[theorem]{Corollary}

\theoremstyle{definition}
\newtheorem{definition}[theorem]{Definition}

\newtheorem{remark}[theorem]{Remark}
\newtheorem{problem}[theorem]{Problem}

\newcommand {\ga}{\gamma}
\newcommand {\Ga}{\Gamma}

\newcommand {\bCP} {\mathbb {CP}}

\newcommand{\Si}{\Sigma}

\newcommand \RG {\mathcal {RG}}
\newcommand \crs {\operatorname{cr}}

\begin{document}
\numberwithin{equation}{section}

\title[Ribbon graphs and meromorphic functions]{Ribbon graphs and meromorphic functions}

\author[B.~Shapiro]{Boris Shapiro}
\address{Department of Mathematics, Stockholm University, SE-106 91
Stockholm, Sweden}
\email{shapiro@math.su.se}

\dedicatory {``The only things worthy of our attention are those which introduce order
\newline in the complexity of the world and so make it accessible to us." H.~Poincar\'e}
\date{}

\begin{abstract}
Let $Y$ be a compact Riemann surface, $\phi:Y\to \bCP^1$ a meromorphic function, and $\Ga\subset Y$ a ribbon graph avoiding the critical points of $\phi$. Then $\phi(\Ga)$ is an immersed graph in $\bCP^1$.

Conversely, given an immersion $im:\Theta\to \bCP^1$ of an abstract multigraph $\Theta$ without vertices of valence $1$ or $2$, we describe a construction of a compact Riemann surface $Y$ and a meromorphic function $\phi_{im}:Y\to \bCP^1$   such that $\phi_{im}(\Ga)=im(\Theta)$.

We further investigate the relation between the topology of $Y$ and the combinatorics of $\Ga$. In particular, for a surface of genus $g$ we construct spanning ribbon graphs whose underlying abstract graphs have arbitrary prescribed graph genus $g'\le g$, including the planar case. As a consequence, the number of self-intersections of $\phi(\Ga)$ cannot, in general, be controlled solely by the genus of $Y$. We establish general lower bounds for the number of self-intersections and formulate several open problems, with emphasis on planar ribbon graphs.
\end{abstract}

\maketitle

\section{Introduction}

In what follows, all abstract multigraphs under consideration are assumed to have no vertices of valence $1$ or $2$.

\begin{definition}
A \emph{ribbon graph} is an abstract multigraph $\Gamma$, possibly with loops and multiple edges, endowed with a cyclic order on the half-edges incident to each vertex.
\end{definition}

Ribbon graphs provide a combinatorial model for studying surfaces and their embeddings. Given an oriented compact Riemann surface $Y$ and an embedded multigraph $\Gamma \subset Y$, the orientation of $Y$ induces a natural ribbon structure on $\Gamma$ via the counterclockwise cyclic ordering at each vertex. Conversely, any abstract ribbon graph $\Gamma$ canonically determines a closed oriented surface $Y_\Gamma$ obtained by thickening $\Gamma$ according to its cyclic structure and capping the resulting boundary components by disks. The genus of $Y_\Gamma$ is called the \emph{genus} of $\Gamma$.

We say that an embedded multigraph $\Gamma \subset Y$ \emph{spans} $Y$ if the complement $Y \setminus \Gamma$ is a disjoint union of topological disks. Spanning ribbon graphs encode the topology of the ambient surface in a particularly efficient way.

In this paper, we investigate the interaction between ribbon graphs and meromorphic functions. Given a compact Riemann surface $Y$ and a meromorphic function $\phi : Y \to \mathbb{CP}^1$, any ribbon graph $\Gamma \subset Y$ avoiding the critical points of $\phi$ gives rise to an immersed graph $\phi(\Gamma) \subset \mathbb{CP}^1$. This construction naturally leads to the problem of understanding which immersed graphs in the sphere arise in this way, and how their combinatorics reflect the topology of $Y$.

Conversely, starting from an immersion $im:\Theta \to \mathbb{CP}^1$ of an abstract multigraph $\Theta$ without vertices of valence $1$ or $2$, one can ask whether this data can be lifted to a triple $(Y, \phi, \Gamma)$ consisting of a compact Riemann surface, a meromorphic function, and a ribbon graph whose image coincides with $im(\Theta)$. One of the main goals of this paper is to establish such a correspondence and to analyze its properties.

A central theme is the relationship between the topology of the surface and the combinatorics of the underlying abstract graph. As shown in Section~\ref{sec:res}, this relationship is highly flexible: for a surface of genus $g$, there exist spanning ribbon graphs whose underlying abstract graphs have any prescribed graph genus $g' \leq g$, and in particular can be chosen to be planar (Proposition~\ref{pr:planar}). This demonstrates that the genus of the surface is not rigidly encoded in the combinatorics of the graph.

This flexibility has important consequences for the study of immersions $\phi(\Gamma) \subset \mathbb{CP}^1$. In particular, it explains why naive expectations relating the number of self-intersections of $\phi(\Gamma)$ to the genus of $Y$ fail in general. Instead, one is led to seek more subtle invariants and lower bounds for the number of self-intersections, taking into account both the ribbon structure and the analytic properties of $\phi$.

The paper is organized as follows. In Section~\ref{sec:res}, we establish the basic existence results for ribbon graphs spanning surfaces with prescribed combinatorial properties and develop the main constructions relating immersed graphs in $\mathbb{CP}^1$ to meromorphic functions. We then investigate lower bounds for the number of self-intersections of such immersions and discuss several open problems, particularly in the case of planar ribbon graphs.

The results of this note are intended as a first step toward a systematic study of immersed ribbon graphs arising from meromorphic functions.

\section{Results and proofs}\label{sec:res}

Our first result concerns the possible underlying abstract graphs of ribbon graphs spanning a given surface.

\begin{proposition}\label{pr:planar}
Given a compact Riemann surface $Y$ of genus $g$, there exist ribbon graphs spanning $Y$ whose underlying abstract graphs have any prescribed graph genus $g'\le g$. In particular, there exist ribbon graphs spanning $Y$ whose underlying abstract graphs are planar. Additionally, one may choose such graphs to be trivalent.
\end{proposition}

\begin{proof}[Proof of Proposition~\ref{pr:planar}]
We first prove the planar case $g'=0$. Consider the standard polygonal presentation of an orientable surface of genus $g$ by a $4g$-gon with side identifications
\[
a_1b_1a_1^{-1}b_1^{-1}\cdots a_gb_ga_g^{-1}b_g^{-1}.
\]
Its $1$-skeleton descends to an embedded bouquet $\Xi_g$ of $2g$ loops based at one vertex. The complement of $\Xi_g$ is a single open disk, so $\Xi_g$ spans $Y$. The underlying abstract graph of $\Xi_g$ is a bouquet of circles and hence planar.

Now fix an integer $g'\le g$. Choose a connected abstract graph $H$ of orientable graph genus $g'$ and fix a cellular embedding of $H$ in a closed orientable surface $\Sigma_{g'}$ of genus $g'$. Let $v$ be a vertex of $H$. Form a new abstract graph $G$ by wedging to $H$ at $v$ a bouquet of $2(g-g')$ loops. Any embedding of $G$ restricts to an embedding of $H$, hence the orientable graph genus of $G$ is at least $g'$. On the other hand, one may embed all added loops inside a small disk neighbourhood of $v$ in the given embedding of $H$, so $G$ also embeds in $\Sigma_{g'}$. Therefore the orientable graph genus of $G$ is exactly $g'$.

Embed $G$ in $Y$ as follows. Realize the chosen embedding of $H$ in a genus-$g'$ subsurface of $Y$, and route each of the $2(g-g')$ added loops through one of the remaining $g-g'$ handles of $Y$ in the standard way coming from the polygonal model above. At this stage the complement need not yet be a union of disks. However, every complementary component is an orientable surface with boundary, and by adding finitely many edges inside each such component one can cut it into disks. These extra edges are attached inside complementary disks or annuli and therefore can be chosen so that the resulting abstract graph still contains $H$ as a subgraph and still embeds in $\Sigma_{g'}$; hence its orientable graph genus remains equal to $g'$. This yields a spanning ribbon graph on $Y$ whose underlying abstract graph has genus $g'$.

Finally, one may make the graph trivalent by the standard local expansion of each vertex of valence $k>3$ into a small $k$-gon, attaching the incident half-edges in the cyclic order determined by the embedding. This operation preserves the ambient surface and the property that the complement is a union of disks. It also preserves the orientable graph genus of the underlying abstract graph: the expanded graph contains the original graph as a minor, so its genus is at least the original one, while collapsing each inserted polygon shows that it embeds in the same surface as before. Hence one can choose the above spanning graphs to be trivalent as well.
\end{proof}

For a given ribbon graph $\Ga\subset Y$, denote by $[\Ga]\ni \Ga$ the orbit of $\Ga$ under the action of the group $\mathrm{Homeo}_0(Y)$ of all homeomorphisms of $Y$ homotopic to the identity. We call $[\Ga]$ the \emph{class} of $\Ga$.

Given a compact orientable Riemann surface $Y$, denote by $\RG_Y$ the set of all ribbon graphs spanning $Y$, i.e.
\[
\RG_Y=\{\Ga\subset Y\mid Y\setminus \Ga \text{ is a finite union of open disks}\}.
\]
Denote by $[\RG_Y]$ the set of classes of ribbon graphs on $Y$.

\begin{remark}
It is easy to show that the minimal elements of $[\RG_Y]$, i.e.
those classes represented by ribbon graphs such that removal of any edge destroys the spanning property, are exactly those represented by graphs $\Ga\subset Y$ satisfying $Y\setminus \Ga$ is a single disk.
\end{remark}

Below we consider two complementary settings, the \emph{direct} and the \emph{inverse} one.

\
\subsection{Direct set-up.}
Consider a compact oriented Riemann surface $Y$ of genus $g$, a meromorphic function $\phi:Y\to \bCP^1$, and a class $[\Ga]$ of ribbon graphs on $Y$. For any $\widetilde \Ga\in[\Ga]$, consider its image $\phi(\widetilde \Ga)\subset \bCP^1$. We say that $\widetilde \Ga$ is \emph{generic} if

\smallskip
\noindent
(i) $\phi(\widetilde \Ga)$ does not contain critical values of $\phi$;

\smallskip
\noindent
(ii) the restriction of $\phi$ to each open edge is an immersion and all singularities of the image are transverse double points coming from two points lying in open edges.

\begin{lemma}\label{lm:generic}
Every class $[\Gamma]$ admits a generic representative.
\end{lemma}

\begin{proof}
Let $\widetilde\Gamma\in[\Gamma]$ be any representative. By a $C^0$-small perturbation of $\phi$ supported away from the vertices, we may assume that all self-intersections are transverse and that no three edges meet at a point outside the vertex set. Such a perturbation does not change the equivalence class $[\Gamma]$. Hence a generic representative exists.
\end{proof}

\smallskip
Observe that if $g(Y)>0$, then for every generic embedded ribbon graph $\Ga\subset Y$, the image $\phi(\Ga)$ must contain intersections of edges. In other words, the map $\phi$ cannot embed a spanning ribbon graph into $\bCP^1$.

Denote by $\sharp(\widetilde \Ga,\phi)$ the number of transverse self-intersections of $\phi(\widetilde \Ga)$. 
We define
\[
\sharp([\Gamma],\phi)
=
\min\{\sharp(\widetilde\Gamma,\phi) \mid \widetilde\Gamma\in[\Gamma] \text{ is generic}\}.
\]

\begin{lemma}\label{lm:minattained}
The minimum defining $\sharp([\Gamma],\phi)$ is attained.
\end{lemma}

\begin{proof}
The number of self-intersections of a generic representative is a nonnegative integer. By Lemma~\ref{lm:generic}, the set over which the minimum is taken is nonempty. Hence the minimum is attained.
\end{proof}

\begin{problem}
Find, or at least estimate from below, $\sharp([\Ga],\phi)$ for a given class $[\Ga]$.
\end{problem}

We start with a basic simplification.

\begin{lemma}\label{lm:noselfint}
For any class $[\Ga]$, there exists a representative $\widetilde \Ga\in[\Ga]$ such that $\phi(\widetilde \Ga)$ has no self-intersections of individual edges. Furthermore, any representative $\Ga^*\in[\Ga]$ realizing the minimum $\sharp([\Ga],\phi)$ has this property.
\end{lemma}

\begin{proof}[Proof of Lemma~\ref{lm:noselfint}]
Let $\widetilde\Gamma\in[\Gamma]$ be a generic representative, which exists by Lemma~\ref{lm:generic}.
 Suppose that some edge $e\subset \widetilde\Ga$ has a self-intersection in the image. Then there exist two distinct interior points $p,q\in e$ such that $\phi(p)=\phi(q)=x$, and since the graph is generic, both points are regular for $\phi$.

Choose pairwise disjoint small discs $U_p,U_q\subset Y$ around $p$ and $q$ such that
\[
\phi|_{U_p}:U_p\to V,\qquad \phi|_{U_q}:U_q\to V
\]
are homeomorphisms onto the same sufficiently small disc $V\subset \bCP^1$ centered at $x$. Inside $V$, the two arcs $\phi(e\cap U_p)$ and $\phi(e\cap U_q)$ meet transversely at $x$. Since $\phi|_{U_p}$ is a homeomorphism, one may isotope the arc $e\cap U_p$ inside $U_p$, fixing its endpoints on $\partial U_p$, so that its image under $\phi$ becomes a nearby arc joining the same endpoints in $V$ but disjoint from $\phi(e\cap U_q)$. Choosing the neighborhoods small enough ensures that no new intersections with the rest of $\phi(\widetilde\Ga)$ are created.

Thus one removes one self-intersection of the image of an edge without changing the class $[\Ga]$. Repeating the procedure finitely many times eliminates all self-intersections of individual edges. In particular, if a representative realizing $\sharp([\Ga],\phi)$ had such a self-intersection, the above move would strictly decrease the number of intersections, a contradiction.
\end{proof}

\begin{proposition}\label{pr:im}
\rm
(i) For any compact Riemann surface $Y$ of genus $g$, any class of ribbon graphs $[\Ga]$, and any meromorphic function $\phi:Y\to \bCP^1$,
\[
\sharp([\Ga],\phi)\ge g.
\]

\smallskip
\noindent
(ii) For any rational function $\phi:\bCP^1\to \bCP^1$ and any class $[\Ga]$,
\[
\sharp([\Ga],\phi)=0.
\]
\end{proposition}

\begin{proof}[Proof of Proposition~\ref{pr:im}]
(i) Choose a representative $\Ga^*\in[\Ga]$ realizing $\sharp([\Ga],\phi)$ and, by Lemma~\ref{lm:noselfint}, assume that no individual edge has self-intersections in the image. Let
\[
d=\sharp([\Ga],\phi)=\sharp(\Ga^*,\phi).
\]
Resolve each transverse double point of the immersed graph $\phi(\Ga^*)\subset \bCP^1$ by inserting a $4$-valent vertex. This produces an embedded graph $H\subset \bCP^1$. Let $N(H)$ be a closed ribbon neighbourhood of $H$ in the sphere. Since $H$ is embedded in $\bCP^1$, the surface obtained from $N(H)$ by capping its boundary components with disks has genus $0$.

Now reconstruct the original immersed image from $H$ one crossing at a time. At a former double point, the passage from the neighbourhood of the $4$-valent vertex in $N(H)$ to the neighbourhood of a transverse crossing is achieved by attaching one $1$-handle to the thickened surface. Consequently, after all $d$ crossings are restored, the closed surface determined by the resulting ribbon graph has genus at most $d$.

On the other hand, the ribbon surface determined by $\Ga^*$ is precisely $Y$, because $\Ga^*$ spans $Y$. Hence
\[
g(Y)=g\le d=\sharp([\Ga],\phi),
\]
which proves (i).

(ii) Let $Y=\bCP^1$ and let $\phi:\bCP^1\to \bCP^1$ be rational. Choose a closed topological disk $D\subset \bCP^1$ avoiding the critical points of $\phi$ and small enough that $\phi|_D$ is injective. Since every orientation-preserving homeomorphism of the sphere is homotopic to the identity, any embedded graph on $\bCP^1$ can be moved by such a homeomorphism into $D$. Hence every class $[\Ga]$ contains a representative $\Ga'\subset D$. For this representative, $\phi(\Ga')$ is embedded, so $\sharp(\Ga',\phi)=0$. Therefore $\sharp([\Ga],\phi)=0$.
\end{proof}

The lower bound in Proposition~\ref{pr:im}(i) is not the whole story. The underlying abstract graph itself may force extra intersections.

\begin{proposition}\label{pr:crossing}
Let $\Ga\subset Y$ be a ribbon graph and let $G$ be its underlying abstract multigraph. Then for any meromorphic function $\phi:Y\to \bCP^1$,
\[
\sharp([\Ga],\phi)\ge \max\{g(Y),\crs(G)\},
\]
where $\crs(G)$ is the planar crossing number of $G$.
\end{proposition}

\begin{proof}[Proof of Proposition~\ref{pr:crossing}]
For any generic representative $\widetilde\Ga\in[\Ga]$, the image $\phi(\widetilde\Ga)$ is an immersion of the underlying abstract multigraph $G$ into $\bCP^1$. Hence it has at least $\crs(G)$ transverse double points by the definition of the planar crossing number. Therefore
\[
\sharp([\Ga],\phi)\ge \crs(G).
\]
Combining this with Proposition~\ref{pr:im}(i), which gives $\sharp([\Ga],\phi)\ge g(Y)$, we obtain the required bound.
\end{proof}

\begin{corollary}\label{cor:mainconjfalse}
The equality $\sharp([\Ga],\phi)=g(Y)$ does not hold for all classes $[\Ga]$ and all meromorphic functions $\phi$.
\end{corollary}

\begin{proof}
Take $Y$ to be a torus and let $\Ga\subset Y$ be an embedded copy of the complete graph $K_7$, which exists since $K_7$ admits a toroidal embedding. Then $g(Y)=1$, while $\crs(K_7)=9$. By Proposition~\ref{pr:crossing}, for every meromorphic function $\phi:Y\to \bCP^1$ one has
\[
\sharp([\Ga],\phi)\ge 9>1=g(Y).
\]
\end{proof}

This suggests the following corrected problem.

\begin{problem}
For which classes $[\Ga]$ does one have $\sharp([\Ga],\phi)=g(Y)$? In particular, is this always true when the underlying abstract graph of $\Ga$ is planar?
\end{problem}

\begin{proposition}\label{pr:necessaryeq}
Let $G$ be the underlying abstract graph of $\Ga\subset Y$. A necessary condition for the equality
\[
\sharp([\Ga],\phi)=g(Y)
\]
to hold is that
\[
\crs(G)\le g(Y).
\]
In particular, the planar underlying-graph case is the natural regime in which one may hope for equality.
\end{proposition}

\begin{proof}
By Proposition~\ref{pr:crossing}, one always has
\[
\sharp([\Ga],\phi)\ge \max\{g(Y),\crs(G)\}.
\]
Therefore, if $\sharp([\Ga],\phi)=g(Y)$, then necessarily $\crs(G)\le g(Y)$.
\end{proof}

\begin{problem}
Are there restrictions on the distribution of the branch points of $\phi$ among the connected components of $\bCP^1\setminus \phi(\widetilde \Ga)$, where $\widetilde \Ga$ is generic?
\end{problem}

\subsection{Inverse set-up.}
Consider an immersion $im:\Theta\to \bCP^1$ of an abstract multigraph $\Theta$. This immersion induces a ribbon structure on $\Theta$ by taking the cyclic orientation of half-edges at each vertex from the orientation of $\bCP^1$ and lifting it back to $\Theta$. Denote by $N(\Theta)$ the corresponding ribbon surface with boundary.

\begin{proposition}\label{lm1}
An immersed graph $im:\Theta\to \bCP^1$ induces a topological branched covering of $\bCP^1$ by a compact Riemann surface $Y$ containing a ribbon graph $\Ga\subset Y$ whose image is $im(\Theta)$.
\end{proposition}

In other words, an immersed graph gives rise to a ribbon graph together with an appropriate meromorphic function.

\begin{lemma}\label{lm:fill}
Given an immersion $im:S^1\to \bCP^1$ together with a choice of one of the two coorientations, there exists a compact connected oriented surface $\Si$ with one boundary component and a topological branched covering
\[
F:\Si\to \bCP^1
\]
whose restriction to $\partial\Si$ is the given immersion $im$, with the chosen boundary coorientation. In general, $\Si$ need not be a disk.
\end{lemma}

\begin{proof}
Orient the immersed circle so that the chosen coorientation is the positive boundary coorientation. Smooth every transverse self-intersection according to this orientation. One obtains a finite disjoint family of embedded oriented circles in $\bCP^1$, the Seifert circles of the immersion. Each Seifert circle bounds a uniquely determined closed disk lying on its positive side; let $D_1,\dots,D_m$ be these disks.

Start with the disjoint union of the disks $D_i$. For every original double point of $im(S^1)$, the smoothing separates the two incoming branches into two boundary arcs belonging to the boundaries of two disks, possibly the same one. Attach a narrow oriented band joining these two arcs exactly as in the usual Seifert-surface construction. After all bands are attached, one obtains a compact connected oriented surface $\Si$ with a single boundary component, naturally identified with the original immersed circle.

The boundary map extends over each disk $D_i$ by the identity map onto its image in $\bCP^1$. Each band is mapped into a small neighborhood of the corresponding double point so that, away from its midpoint, the map is a local homeomorphism onto one of the two branches, while at the midpoint the two sheets come together with the standard local model of a simple branch point. In local complex coordinates this model is equivalent to $z\mapsto z^2$. Choosing the bands pairwise disjoint and sufficiently small, these local extensions patch together to a continuous map
\[
F:\Si\to \bCP^1
\]
which is open and discrete, is a local homeomorphism away from finitely many interior points, and has only simple branching at those points. Hence $F$ is a topological branched covering extending the given boundary immersion.

The surface $\Si$ need not be a disk: already for immersed circles with sufficiently many crossings, the Seifert construction may produce positive genus.
\end{proof}

\begin{proof}[Proof of Proposition~\ref{lm1}]
Given an immersed graph $im:\Theta\to \bCP^1$, endow $\Theta$ with the ribbon structure induced from the orientation of $\bCP^1$. Let $N(\Theta)$ be the corresponding ribbon surface with boundary. By construction, the immersion of $\Theta$ extends to a map from a neighborhood of $\Theta$ in $N(\Theta)$ that is a local homeomorphism.

Each boundary component $C\subset \partial N(\Theta)$ maps to an immersed circle in $\bCP^1$ and inherits a natural coorientation from the surface $N(\Theta)$. By Lemma~\ref{lm:fill}, for each such component there exists a compact connected oriented surface $\Si_C$ with one boundary component and a topological branched covering
\[
F_C:\Si_C\to \bCP^1
\]
extending the boundary immersion.

Glue all surfaces $\Si_C$ to $N(\Theta)$ along the corresponding boundary components. The local models along the gluing circles agree by construction, so the resulting map
\[
\phi:Y\to \bCP^1
\]
on the closed oriented surface $Y$ is open and discrete and is locally equivalent either to a homeomorphism or to $z\mapsto z^k$ at finitely many points. Thus $\phi$ is a topological branched covering. By Sto\"{\i}low's theorem, there exists a unique complex structure on $Y$ for which $\phi$ becomes meromorphic. The embedded copy of $\Theta$ inside $Y$ is a ribbon graph whose image under $\phi$ is exactly $im(\Theta)$.
\end{proof}

\begin{problem}
Given an immersed graph $im:\Theta\to \bCP^1$, describe all other immersions of $\Theta$ into $\bCP^1$ inducing the same ribbon structure. The same question can be asked for a fixed abstract ribbon graph.
\end{problem}

It seems plausible that all such immersions can be related by Reidemeister moves together with creation and annihilation of loops, but we do not pursue this here.

\begin{problem}
Given an immersed graph, how does one compute the degree of the corresponding covering? Is it true that the degree determines the number of branch points?
\end{problem}

\begin{remark}
The degree does not determine the number of branch points. Indeed, Riemann--Hurwitz fixes the total ramification
\[
\sum_{p\in Y}(e_p-1)=2\deg\phi+2g(Y)-2,
\]
but does not determine how this ramification is distributed among distinct branch values.
\end{remark}

\begin{proposition}\label{pr:degbranch}
The degree of the branched covering associated with an immersed graph does not determine the number of branch points.
\end{proposition}

\begin{proof}
This is an immediate consequence of Riemann--Hurwitz. The degree and genus determine only the total ramification
\[
\sum_{p\in Y}(e_p-1)=2\deg\phi+2g(Y)-2,
\]
and not the number of distinct branch values in $\bCP^1$. The same total ramification may be distributed among many simple branch values or concentrated over fewer branch values.
\end{proof}

\begin{problem}
Given an immersed graph, how does one describe all restrictions on the location of the branch points of the corresponding covering?
\end{problem}

\begin{proposition}\label{pr:branchregion}
Let $\widetilde\Ga$ be a generic representative of a ribbon graph class and let $U$ be a connected component of
\[
\bCP^1\setminus \phi(\widetilde\Ga).
\]
Choose a closed disk $\overline U\subset \bCP^1$ with interior $U$ and boundary equal to the topological boundary of $U$. For each connected component $S$ of $\phi^{-1}(\overline U)$, the restriction
\[
\phi|_S:S\to \overline U
\]
is a branched covering of compact oriented surfaces with boundary and satisfies
\[
\chi(S)=d_S-\sum_{p\in \operatorname{int}(S)}(e_p-1),
\]
where $d_S=\deg(\phi|_S)$. In particular, if $S$ is a disk and the boundary map has degree $1$, then $\phi|_S$ is unbranched, so $U$ contains no branch values coming from $S$.
\end{proposition}

\begin{proof}
Because $\widetilde\Ga$ is generic, the boundary of $U$ contains no critical values of $\phi$. Hence every connected component $S$ of $\phi^{-1}(\overline U)$ is a compact oriented surface with boundary, and the restriction $\phi|_S$ is a branched covering of degree $d_S$, all of whose branch points lie in $\operatorname{int}(S)$.

Applying the Riemann--Hurwitz formula for branched coverings of compact surfaces with boundary gives
\[
\chi(S)=d_S\chi(\overline U)-\sum_{p\in \operatorname{int}(S)}(e_p-1)=d_S-\sum_{p\in \operatorname{int}(S)}(e_p-1),
\]
since $\chi(\overline U)=1$.

If $S$ is a disk and the boundary map has degree $1$, then necessarily $d_S=1$. The above formula then gives
\[
1=1-\sum_{p\in \operatorname{int}(S)}(e_p-1),
\]
so the ramification sum is zero. Therefore $\phi|_S$ is unbranched, and $U$ contains no branch values contributed by $S$.
\end{proof}

For planar trivalent ribbon graphs there is another way to encode their ribbon structure in $\bCP^1$ besides using immersed images. This may substantially reduce the number of intersections of edges. We call an immersed graph \emph{2-coloured} if its vertices are coloured black and white. The ribbon structure induced by a 2-coloured trivalent immersed graph is defined by taking the clockwise cyclic order at white vertices and the counterclockwise cyclic order at black vertices.

\begin{problem}
Given a trivalent ribbon graph $\Ga\subset Y$ of genus $g$, find lower bounds on the number of intersections of edges if $\Ga$ is realized in $\bCP^1$ as a 2-coloured immersed graph.
\end{problem}

\begin{proposition}\label{pr:2colorbound}
Let $\Ga$ be a trivalent ribbon graph of genus $g(\Ga)$. Any 2-coloured immersed realization of its underlying abstract graph in $\bCP^1$ inducing the given ribbon structure has at least $g(\Ga)$ transverse edge intersections. More generally, if $G$ is the underlying abstract graph, then the number $d$ of transverse edge intersections satisfies
\[
d\ge \max\{g(\Ga),\crs(G)\}.
\]
\end{proposition}

\begin{proof}
Replace each transverse intersection of edges by a new $4$-valent vertex. This produces an embedded graph $H\subset \bCP^1$, and the ribbon surface associated with $H$ has genus $0$. Undoing one such replacement restores a crossing and can increase the genus of the thickened surface by at most $1$. After all $d$ crossings are restored, one recovers the original ribbon graph $\Ga$, hence
\[
g(\Ga)\le d.
\]
On the other hand, any immersed realization of the underlying abstract graph $G$ in the sphere has at least $\crs(G)$ double points by definition of the planar crossing number. Combining the two bounds yields the claim.
\end{proof}

\begin{remark}
Any planar trivalent ribbon graph can be realized as an embedded 2-coloured graph.
\end{remark}

Now let $\ga\subset \bCP^1$ be an embedded trivalent planar graph with black and white vertices. To such a graph we associate a ribbon graph $\widehat \ga$ by taking the clockwise cyclic order at white vertices and the counterclockwise cyclic order at black vertices. More generally, for any trivalent ribbon graph one may ask for criteria guaranteeing uniqueness of an embedding into a surface of minimal genus.

Given a planar ribbon graph $\Ga$ and an embedding $em:\Ga\hookrightarrow \bCP^1$, we get two cyclic orders at the vertices of $\Ga$: the original one and the one induced by the planar embedding. By the \emph{defect} $\operatorname{def}(\Ga,em)$ of the planar embedding we mean the number of vertices at which these cyclic orders differ.

\begin{problem}
How does one find embeddings $\Ga\hookrightarrow \bCP^1$ with minimal defect?
\end{problem}

\begin{proposition}\label{pr:defect3conn}
Let $\Ga$ be a planar $3$-connected graph with a fixed ribbon structure. Then the minimal defect of a planar embedding equals
\[
\min\{\delta_+(\Ga),\delta_-(\Ga)\},
\]
where $\delta_+(\Ga)$ and $\delta_-(\Ga)$ are the numbers of vertices at which the prescribed cyclic order disagrees with, respectively, one planar embedding and its mirror image.
\end{proposition}

\begin{proof}
By Whitney's theorem, a $3$-connected planar graph has exactly two planar embeddings up to orientation-preserving homeomorphism of the sphere, namely one embedding and its mirror image. Thus there are only two planar rotation systems to compare with the prescribed ribbon structure. At each vertex the contribution to the defect is $0$ or $1$ depending on whether the cyclic orders agree. Therefore the smallest possible defect is exactly the smaller of the two disagreement counts.
\end{proof}

Recall that all planar embeddings of a given planar graph are constrained by Whitney-type results; in particular, each 3-connected planar graph has a unique embedding up to mirror symmetry.

For trivalent planar graphs the latter problem is especially natural. A trivalent vertex has only two cyclic orientations, which we call black and white.

\begin{problem}\label{pr:genusbw}
Given an embedded planar graph $em:\Ga\to \bCP^1$ with black and white vertices, what is the genus of the oriented surface carrying the associated ribbon graph?
\end{problem}

The answer is determined by the induced rotation system.

\begin{proposition}\label{pr:bwgenus}
Let $em:\Ga\hookrightarrow \bCP^1$ be an embedded planar trivalent graph with black and white vertices. Let $\sigma$ be the permutation of half-edges obtained by taking the counterclockwise cyclic order at black vertices and the clockwise cyclic order at white vertices, and let $\alpha$ be the fixed-point-free involution pairing the two half-edges of each edge. If $F_\sigma$ denotes the number of cycles of the permutation $\sigma\alpha$, then the genus of the associated ribbon graph is
\[
g(\Ga)=1-\frac{V-E+F_\sigma}{2}.
\]
For a trivalent graph one has $E=\frac{3V}{2}$, hence
\[
g(\Ga)=1+\frac{V}{4}-\frac{F_\sigma}{2}.
\]
\end{proposition}

\begin{proof}
The permutation $\sigma\alpha$ is the face permutation of the ribbon thickening determined by the chosen black/white cyclic orders, so the number of its cycles is exactly the number $F_\sigma$ of boundary components of that thickening. The Euler characteristic of the corresponding closed oriented surface is therefore
\[
\chi=V-E+F_\sigma,
\]
and the genus is given by
\[
g(\Ga)=1-\frac{\chi}{2}=1-\frac{V-E+F_\sigma}{2}.
\]
Substituting $E=\frac{3V}{2}$ for a trivalent graph yields the second formula.
\end{proof}

\medskip
We conclude this section with a family of related questions inspired by local reflections of black-and-white graphs. These problems appear interesting, but unlike the statements above, they remain speculative.

\begin{problem}
What is the effect on the associated branched covering of reflecting a topological disk containing only white vertices of a 2-coloured planar graph and reglueing it so that those vertices become black? Can one estimate the resulting increase in the number of branch points in terms of the number of cut edges?
\end{problem}

\begin{problem}
Is the number of branch points independent of the order in which one performs such disk-reflection operations?
\end{problem}

\begin{problem}
For which planar graphs with black and white vertices do the branch points obtained from an intrinsic branched covering construction coincide with those predicted by a sequence of disk-reflection operations? We call such graphs \emph{minimal}.
\end{problem}

\section{Final remarks}

The preceding discussion shows that the direct and inverse constructions are robust, but also that the naive expectation $\sharp([\Ga],\phi)=g(Y)$ is too optimistic without further restrictions. Several of the open questions admit partial answers: the degree does not determine the number of branch points; the genus of a black-and-white planar trivalent graph is computed by its rotation system; the defect minimization problem is completely solved for $3$-connected planar graphs; and any $2$-coloured immersed realization must have at least $\max\{g(\Ga),\crs(G)\}$ transverse edge intersections. The natural next step seems to be the planar underlying-graph case, where the lower bound coming from the planar crossing number disappears and equality with the genus becomes plausible.

\medskip\noindent
\textbf{Acknowledgements.}
The author is sincerely grateful to Professor A.~Gabrielov of Purdue University for explaining  his experimental observations related to this topic about a decennium ago, see \cite{ErGa}.


\begin{thebibliography}{30}

\bibitem{ChGa} Zhi Yun Cheng, Hong Zhu Gao, Mutation on Knots and Whitney's 2-Isomorphism Theorem, \emph{Acta Mathematica Sinica, English Series} \textbf{29} (2013), no.~6, 1219--1230.

\bibitem{ErGa} A.~Eremenko and A.~Gabrielov, Irreducibility of some spectral determinants, arXiv:0904.1714.

\bibitem{MuPe} Motohico Mulase, Michael Penkava, Ribbon Graphs, Quadratic Differentials on Riemann Surfaces, and Algebraic Curves Defined over $\overline{\mathbb Q}$, \emph{Asian Journal of Mathematics} \textbf{2} (1998), no.~4, 875--920.

\end{thebibliography}
\end{document}